%% file: deviations.tex
\documentclass[11pt]{imsart}

\input{header}

\startlocaldefs

\input{macros}
\endlocaldefs

\date{\today}

\begin{document}

\title{Localized Gaussian width of $M$-convex hulls with applications to Lasso and convex aggregation}
\author{Pierre C. Bellec}
\address{Rutgers University, Department of Statistics and Biostatistics}

\begin{abstract}
    Upper and lower bounds are derived
    for the Gaussian mean width of the intersection
    of a convex hull of $M$ points with an Euclidean ball of a given radius.
    The upper bound holds for any collection of extreme point bounded in Euclidean norm.
    The upper bound and the lower bound match up to a multiplicative constant whenever the extreme points satisfy a one sided Restricted Isometry Property.

    This bound is then applied to study the Lasso estimator in fixed-design regression,
    the Empirical Risk Minimizer in the anisotropic persistence problem,
    and the convex aggregation problem in density estimation.
\end{abstract}

\maketitle



\section{Introduction}\label{introduction}

Let $T$ be a subset of $\Rn$.
The Gaussian width of $T$ is defined as
\begin{equation}
    \ell(T) \coloneqq  \E \sup_{\vu\in T} \vu^T \vg, 
\end{equation}
where $\vg = (g_1,...,g_n)^T$ and $g_1,...,g_n$ are i.i.d. standard normal random variables.
For any vector $\vu\in\Rn$, denote by $|\vu|_2$ its Euclidean norm and define
the Euclidean balls
\begin{equation}
    B_2 = \{ \vu\in\Rn: \euclidnorm{\vu}\le 1 \},
    \qquad
    s B_2 = \{ s \vu\in\Rn,\; \vu\in B_2  \}
    \text{ for all } s\ge0.
\end{equation}
We will also use the notation $S^{n-1} = \{\vu\in\R^n:\euclidnorm{\vu}=1 \}$.
The localized Gaussian width of $T$ with radius $s>0$ is the quantity
$\ell(T\cap sB_2)$.
For any $\vu\in\R^p$, define the $\ell_p$ norm by $|\vu|_p=(\sum_{i=1}^n |u_i|^p )^{1/p}$ for any $p\ge1$,
and let $|\vu|_0$ be the number of nonzero coefficients of $\vu$.

This paper studies the localized Gaussian width
\begin{equation}
    \ell(sB_2\cap T),
\end{equation}
where $T$ is the convex hull of $M$ points in $\R^n$.

If $T=B_1=\{\vu\in\R^n: \onenorm{\vu}\le 1 \}$,
then matching upper and lower bounds are available for 
the localized Gaussian width:
\begin{equation}
    \ell(sB_2\cap B_1)
    \asymp
        \sqrt{  \log\left(e n (s^2 \wedge 1) \right) }
    \wedge
    \left(
        s \sqrt{n}
    \right),
    \label{eq:diagonal-case-mendelson}
\end{equation}
cf. \cite{gordon2007gaussian} and \cite[Section 4.1]{lecue2013learning}.
In the above display, $a\asymp b$ means that $a\le C b$ and $b\le C a$
for some large enough numerical constant $C\ge 1$.

The first goal of this paper is to generalize this bound
to any $T$ that is the convex hull of $M\ge1$ points
in $\R^n$.

\paragraph{Contributions.}

\Cref{s:expected-sup} is devoted to the generalization of \eqref{eq:diagonal-case-mendelson} and provides sharp bounds on the localized Gaussian width of the convex hull of $M$ points in $\R^n$, see \Cref{prop:gw-simple,prop:lower-bound-RIP} below.
\Cref{s:applications-fixed-design,s:persistence,s:bounded} provide
statistical applications of the results of \Cref{s:expected-sup}.
\Cref{s:applications-fixed-design} studies 
the Lasso estimator and the convex aggregation problem in
fixed-design regression.
In \Cref{s:persistence}, we show that Empirical Risk Minimization achieves the minimax rate for the persistence problem
in the anisotropic setting.
Finally, \Cref{s:bounded} provides results for bounded empirical processes
and for the convex aggregation problem in density estimation.

\section{Localized Gaussian width of a $M$-convex hull
    \label{s:expected-sup}
}

The first contribution of the present paper is the following upper bound
on localized Gaussian width of the convex hull of $M$ points in $\Rn$.
\begin{prop}
    \label{prop:gw-simple}
    Let $n\ge1$ and $M\ge2$.
    Let $T$ be the convex hull of $M$ points in $\Rn$ and assume that $T\subset B_2$.
    Let $\vg$ be a centered Gaussian random variable with covariance matrix $I_{n\times n}$.
    Then for all $s>0$,
    \begin{equation}
        \ell(T \cap s B_2 ) 
        \le  
        \left(
            4 \sqrt{  \log_+\left( 4e M (s^2 \wedge 1) \right) }
        \right)
        \wedge
        \left(
            s \sqrt{n \wedge M}
        \right)
        \label{eq:full-upper-bound}
    \end{equation}
    where $\log_+(a) = \max(1, \log a)$.
\end{prop}
\Cref{prop:gw-simple} is proved in the next two subsections.
Inequality 
\begin{equation}
    \ell(T \cap s B_2 ) \le s\sqrt{n\wedge M}
    \label{eq:bound-easy-cs}
\end{equation} is a direct consequence of the Cauchy-Schwarz inequality
and $\E\euclidnorm{P \vg }\le \sqrt{d}$ where $P\in\R^{n\times n}$
is the orthogonal projection onto the linear span of $T$ and $d\le(n \wedge M)$ is the rank of $P$.
The novelty of \eqref{eq:full-upper-bound} is inequality
\begin{equation}
    \ell(T \cap s B_2 )
    \le  4 \sqrt{  \log_+\left( 4 eM (s^2\wedge1) \right) }
    .
    \label{eq:gaussian-width-s}
\end{equation}
Inequality \eqref{eq:gaussian-width-s} was known for the $\ell_1$-ball $T=\{\vu\in\R^n: \onenorm{\vu}\le 1 \}$ \cite{gordon2007gaussian},
but to our knowledge \eqref{eq:gaussian-width-s} is new for general $M$-convex hulls.
If $T$ is the $\ell_1$-ball, then the bound \eqref{eq:full-upper-bound}
is sharp up to numerical constants \cite{gordon2007gaussian}, \cite[Section 4.1]{lecue2013learning}.

The above result does not assume any type of Restricted Isometry Property (RIP).
The following proposition shows that \eqref{eq:gaussian-width-s}
is essentially sharp provided that the vertices of $T$ satisfies a one-sided RIP
of order $2/s^2$.

\begin{prop}
    \label{prop:lower-bound-RIP}
    Let $n\ge1$ and $M\ge2$.
    Let $\vg$ be a centered Gaussian random variable with covariance matrix $I_{n\times n}$.
    Let $s\in(0,1]$ and assume for simplicity that $m=1/s^2$ is a positive integer
    such that $m \le M/5$.
    Let $T$ be the convex hull of the $2M$ points $\{\pm\vmu_1,...,\pm\vmu_M\}$ where
    $\vmu_1,...,\vmu_M\in S^{n-1}$.
    Assume that
    for some real number $\kappa\in(0,1)$ we have
    \begin{equation}
         \kappa |\vtheta|_2 \le |\vmu_\vtheta|_2
         \qquad
         \text{ for all }
         \vtheta\in\R^M
         \text{ such that }
        |\vtheta|_0\le 2m,
    \end{equation}
    where
    $\vmu_\vtheta= \sum_{j=1}^M \theta_j\vmu_j$.
    Then
    \begin{equation}
        \ell(T \cap s B_2 ) 
        \ge (\sqrt 2/4) \kappa \sqrt{ \log\left(\frac{Ms^2}{5}\right)}
        .
        \label{eq:lower-bound-RIP}
    \end{equation}
\end{prop}
The proof of \Cref{prop:lower-bound-RIP} is given in \Cref{s:proof-lower-RIP}.

\subsection{A refinement of Maurey's argument}
This subsection provides the main tool to derive
the upper bound \eqref{eq:gaussian-width-s}.
Define the simplex in $\R^M$ by
\begin{equation}
    \simplex = \Big\{
        \vtheta \in \RM, 
        \quad \sum_{j=1}^M \theta_j = 1,
        \quad \forall j=1\dots M, \;\; \theta_j \ge 0
    \Big\}.
    \label{eq:def-simplex}
\end{equation}
Let $m\ge 1$ be an integer,
and let
\begin{equation}
Q(\vtheta)= \vtheta^T \Sigma \vtheta
,
\end{equation}
where
$\Sigma = (\Sigma_{jj'})_{j,j'=1,...,M}$
is a positive semi-definite matrix of size $M$.
Let $\bar\vtheta\in\simplex$ be a deterministic vector
such that $Q(\bar\vtheta)$ is small.
Maurey's argument \cite{pisier1980remarques}
has been used extensively to prove the existence
of a sparse vector $\tilde\vtheta\in\simplex$
such that $Q(\tilde\vtheta)$ is of the same order as that of $Q(\bar\vtheta)$.
Maurey's argument uses the probabilistic method to prove the existence of such $\tilde\vtheta$.
A sketch of this argument is as follows.

Define
the discrete set $\simplex_m$ as
\begin{equation}
    \label{eq:cgrid}
    \simplex_m \coloneqq \left\{ \frac{1}{m}\sum_{k=1}^m u_k,\;\; \vu_1,..., \vu_m\in \{\ve_1,...,\ve_M\} \right\},
\end{equation}
where $(\ve_1,...,\ve_M)$ is the canonical basis in $\RM$.
The discrete set $\simplex_m$ is a subset of the simplex $\simplex$
that contains only $m$-sparse vectors.

Let $(\ve_1,...,\ve_M)$ be the canonical basis in $\RM$.
Let $\Theta_1,...,\Theta_m$ be i.i.d. random variables 
valued in $\{\ve_1,...,\ve_M\}$ with distribution
\begin{equation}
    \proba{\Theta_k=\ve_j}=\bar\theta_j
    \qquad
    \text{ for all }k=1,...,m.
    \label{eq:def-Theta_i}
\end{equation}
Next, consider the random variable 
\begin{equation}
    \that = \frac{1}{m}\sum_{k=1}^m \Theta_k.
    \label{eq:def-that-maurey}
\end{equation}
The random variable $\hat\vtheta$ is valued in $\simplex_m$
and is such that $\E_{\Theta}[ \hat\vtheta ]= \bar\vtheta$, where $\E_\Theta$ denotes the expectation with respect to $\hat\vtheta$.
Then a bias-variance decomposition yields
\begin{equation}
    \label{eq:control-Q}
    \E_\Theta[\hat\vtheta] \le Q(\bar\vtheta) + R^2 / m,
\end{equation}
where $R>0$ is a constant such that $\max_{j=1,...,M}\Sigma_{jj} \le R^2$.
As $\min_{\vtheta\in\simplex_m} Q(\vtheta) \le \E_\Theta[\hat\vtheta]$,
this yields the existence of $\tilde\vtheta\in\simplex_m$ such that
\begin{equation}
    Q(\tilde\vtheta) \le Q(\bar\vtheta) + R^2 /m.
\end{equation}
If $m$ is chosen large enough, the two terms $Q(\bar\vtheta)$ and $R^2/m$ are of the same order and we have established the existence of an $m$-sparse vector $\tilde\vtheta$ so that $Q(\tilde\vtheta)$ is not much
substantially larger than $Q(\bar\vtheta)$.

For our purpose, we need to refine this argument
by controlling the deviation of the
random variable $Q(\that)$.
This is done in \Cref{lemma:maureysup} below.

\begin{lemma}
    \label{lemma:maureysup}
    Let $m\ge 1$ and define $\simplex_m$ by \eqref{eq:cgrid}.
    Let $F:\RM\rightarrow [0,+\infty)$ be a convex function.
    For all $\vtheta\in\RM$,  let
    \begin{equation}
    Q(\vtheta)= \vtheta^T \Sigma \vtheta
    ,
    \end{equation}
    where
    $\Sigma = (\Sigma_{jj'})_{j,j'=1,...,M}$ is a positive semi-definite matrix of size $M$.
    Assume that the diagonal elements of $\Sigma$ satisfy
    $\Sigma_{jj} \le R^2$   for all $j=1,...,M$.
    Then for all $t > 0$,
    \begin{equation}
        \sup_{\vtheta\in\simplex: Q(\vtheta)\le t^2}
        F(\vtheta)
        \le
        \int_1^{+\infty}
        \left[
            \max_{\vtheta\in\simplex_m:\; Q(\vtheta) \le x(t^2+ R^2/m)}
            F(\vtheta)
        \right] \frac{dx}{x^2}.
        \label{eq:maureysup}
    \end{equation}
\end{lemma}

In the next sections, it will be useful to bound from above
the quantity 
$F(\vtheta)$ maximized over $\simplex$ subject to the constraint $Q(\vtheta)\le t^2$.
An interpretation of \eqref{eq:maureysup} is as follows.
Consider the two optimization problems
\begin{align}
    \text{maximize }
    F(\vtheta)
    &&\text{for }\vtheta\in\simplex 
    &&\text{subject to } Q(\vtheta)\le t^2,
\\
    \text{maximize }
    F(\vtheta)
    &&\text{for }\vtheta\in\simplex_m
    &&\text{subject to } Q(\vtheta)\le Y(t^2+R^2/m),
\end{align}
for some $Y\ge1$.
\Cref{eq:maureysup} says that
the optimal value of
the first optimization problem
is smaller than
the optimal value of the second optimization problem
averaged over the distribution of $Y$ given by the density
$y\mapsto 1/y^2$ on $[1,+\infty)$.
The second optimization problem above is
over the \emph{discrete} set $\simplex_m$ with 
the relaxed constraint $Q(\vtheta)\le Y(t^2+R^2/m)$,
hence we have relaxed the constraint in exchange for discreteness.
The discreteness of the set $\simplex_m$ will be used in the next subsection
for the proof of \Cref{prop:gw-simple}.

\begin{proof}[Proof of \Cref{lemma:maureysup}]
    The set $\{\vtheta\in\simplex: Q(\vtheta) \le t^2\}$ is compact.
    The function $F$ is convex with domain $\R^M$ and thus continuous.
    Hence the supremum in the left hand side of \eqref{eq:maureysup}
    is achieved at some
    $\bar \vtheta\in\simplex$  such that
    $Q(\bar\vtheta) \le t^2$.
    Let $\Theta_1,...,\Theta_m,\that$ be the random variable defined in
    \eqref{eq:def-Theta_i} and \eqref{eq:def-that-maurey} above.
    Denote by $E_\Theta$ the expectation with respect to $\Theta_1,...,\Theta_m$.
    By definition, $\that\in\simplex_m$ and $\E_\Theta \that = \bar\vtheta$. 
    Let $E = \E_\Theta[ Q(\that)]$.
    A bias-variance decomposition and the independence of $\Theta_1,...,\Theta_m$ yield
    \begin{align}
        E \coloneqq \E_\Theta[ Q(\that)] 
        &= Q(\bar\vtheta) + \E_\Theta (\that-\bar\vtheta)^T\Sigma (\that -\bar\vtheta), \\
        &= Q(\bar\vtheta) + \frac{1}{m} \E_\Theta [ (\Theta_1-\bar\vtheta)^T\Sigma (\Theta_1-\bar\vtheta) ].
    \end{align}
    Another bias-variance decomposition yields
    \begin{equation}
        \E_\Theta (\Theta_1-\bar\vtheta)^T\Sigma (\Theta_1-\bar\vtheta)
        =
        \E_\Theta [ Q(\Theta_1) ]
        - Q(\bar\vtheta)
        \le
        \E_\Theta Q(\Theta_1)
        \le 
        R^2,
    \end{equation}
    where we used that $Q(\cdot)\ge0$ and that
$\Theta_1 \Sigma\Theta_1 \le R^2$ almost surely.
Thus 
\begin{equation}
    E = \E_\Theta [Q(\that)] \le Q(\bar\vtheta) + R^2/m \le t^2 + R^2/m.
    \label{eq:bound-E}
\end{equation}
Define the random variable $X=Q(\that)/E$, which is nonnegative and satisfifes $\E_\Theta[X] = 1$.
By Markov inequality, it holds that $\mathbb P_\Theta(X>t)\le 1/t = \int_1^{+\infty}(1/x^2)dx$.
Define the random variable $Y$ by the density function $x\rightarrow 1/x^2$ on $[1,+\infty)$.
Then we have $\mathbb P_\Theta(X>t) \le \mathbb P(Y>t)$ for any $t>0$,
so by stochastic dominance, there exists a rich enough probability space $\Omega$
and random variables $\tilde X$ and $\tilde Y$ defined on $\Omega$ such that
$\tilde X$ and $X$ have the same distribution,
$\tilde Y$ and $Y$ have the same distribution,
and $\tilde X \le \tilde Y$ almost surely on $\Omega$
(see for instance Theorem 7.1 in \cite{hollander2012probability}).
Denote by $\E_{\Omega}$ the expectation sign on the probability space $\Omega$.

By definition of $\bar\vtheta$ and $\that$,
using Jensen's inequality,
Fubini's Theorem and the fact that $\that\in\simplex_m$ we have
\begin{align}
    \sup_{\vtheta\in\simplex: Q(\vtheta)\le t^2}
    F(\vtheta)
    = F(\bar\vtheta)
    = F(\E_\Theta[\that])
    \le \E_\Theta [F(\that)]
    \le
        \E_\Theta[
        g(Q(\that)/E)
        ]
\end{align}
where $g(\cdot)$ is the nondecreasing function $g(x)= \max_{\vtheta\in\simplex_m: Q(\vtheta)\le x E} F(\vtheta)$.
The right hand side of the previous display is equal to to $\E_\Theta[g(X)]$.
Next, we use the random variables $\tilde X$ and $\tilde Y$ as follows:
\begin{equation}
    \E_\Theta[g(X)]
    = \E_\Omega[g(\tilde X)]
    \le
    \E_\Omega[g(\tilde Y)]
    = \int_1^{+\infty} \frac{g(x)}{x^2} dx.
\end{equation}
Combining the previous display and
\eqref{eq:bound-E}
completes the proof.
\end{proof}

\subsection{Proof of \eqref{eq:gaussian-width-s}}
We are now ready to prove \Cref{prop:gw-simple}.
The main ingredients are \Cref{lemma:maureysup}
and the following upper bound on the cardinal of $\simplex_m$
\begin{equation}
    \label{eq:bound-cardinal}
    \log |\simplex_m|
    =
    \log {M + m - 1 \choose m}
    \le
    \log {2M \choose m}
    \le m \log \left( \frac{2eM}{m} \right).
\end{equation}

\begin{proof}[Proof of \eqref{eq:gaussian-width-s}] 
    If $s^2<1/M$ then by \eqref{eq:bound-easy-cs} we have $\ell(T\cap s B_2)\le 1$, hence \eqref{eq:gaussian-width-s} holds.
    Thus it is enough to focus on the case $s^2\ge 1/M$.

    Let $r=\min(s,1)$ and set
    $m = \lfloor 1/r^2 \rfloor$, which satisfies $1\le m \le M$.
    As $T$ is the convex hull of $M$ points,
    let $\vmu_1,...,\vmu_M \in \Rn$ be such that
    \begin{equation}
        T = \text{convex hull of } \{\vmu_1,...,\vmu_M\} 
        = \{ \vmu_\vtheta, \vtheta\in\simplex \},
    \end{equation}
    where $\vmu_\vtheta = \sum_{j=1}^M \theta_j \vmu_j$ for $\vtheta\in\simplex$.
    
Let $Q(\vtheta) = |\vmu_\vtheta|_2^2$ for all $\vtheta\in\RM$.
This is a polynomial of order $2$, of the form $Q(\vtheta)=\vtheta^T \Sigma\vtheta$,
where $\Sigma$ is the Gram matrix with $\Sigma_{jk} = \vmu_k^T\vmu_j$
for all
$j,k=1,...,M$.
As we assume that $T\subset B_2$, the diagonal elements of $\Sigma$ satisfy $\Sigma_{jj} \le 1$.
For all $\vtheta\in\RM$, let $F(\vtheta) = \vg^T\vmu_\vtheta$.
Applying \Cref{lemma:maureysup} with the above notation, $R=1$,
$m = \lfloor 1/r^2 \rfloor$ and
$t = r$,
we obtain
\begin{equation}
    \E \sup_{\vtheta\in\simplex: Q(\vtheta)\le r^2} 
    \vg^T \vmu_\vtheta
        \le
        \E
        \int_1^{+\infty}
        \left[
            \max_{\vtheta\in\simplex_m:\; Q(\vtheta) \le x(r^2 + 1/m)}
            F(\vtheta)
        \right] \frac{dx}{x^2}.
\end{equation}
By definition of $m$, $r^2 \le 1/m$ so that $x(r^2+1/m) \le 2x/m$.
Using Fubini Theorem and a bound on 
the expectation of the maximum of $|\simplex_m|$ centered Gaussian random variables
with variances bounded from above by $2x/m$,
we obtain that the right hand side of the previous display is bounded from above by
\begin{equation}
    \int_1^{+\infty}
    \frac{1}{x^2}
    \sqrt{\frac{4x \log|\simplex_m|}{m}}
    dx
    \le
    \sqrt{\log(2eM/m)}
    \int_1^{+\infty}
    \frac{2}{x^{3/2}}
    dx.
\end{equation}
where we used the bound \eqref{eq:bound-cardinal}.
To complete the proof of \eqref{eq:gaussian-width-s}, notice that we have
$1/m \le 2 r^2$ and
$
    \int_1^{+\infty}
    \frac{2}{x^{3/2}}
    dx
    = 4
$.

\end{proof}

\section{Statistical applications in fixed-design regression}
\label{s:applications-fixed-design}

Numerous works have established a close relationship between
localized Gaussian widths
and the performance of 
statistical and compressed sensing procedures.
Some of these works are reviewed below.

\begin{itemize}
    \item In a regression problem with random design where
        the design and the target are subgaussian,
        \citet{lecue2013learning} established that
        two quantities govern the performance of empirical risk minimizer over
        a convex class $\mathcal F$.
        These two quantities are defined using
        the Gaussian width of the class $\mathcal F$
        intersected with an $L_2$ ball \cite[Definition 1.3]{lecue2013learning},

    \item If $p,p'>1$ are such that $p'\le p \le +\infty$ and $\log(2n)/(\log(2en)\le p'$.
        \citet{gordon2007gaussian} provide precise estimates of $\ell(B_p\cap s B_{p'})$
        where $B_p \subset \R^n$ is the unit $L_p$ ball
        and $s B_{p'}$ is the $L_{p'}$ ball of radius $s>0$.
        These estimates are then used to solve the approximate reconstruction problem
        where one wants to recover an unknown high dimensional vector from
        a few random measurements \cite[Section 7]{gordon2007gaussian}.

    \item \citet{plan2014high} shows that in the semiparametric single index model, if
        the signal is known to belong to some star-shaped set $T\subset\R^n$, then the
        Gaussian width of $T$ and its localized version characterize the gain obtained
        by using the additional information that the signal belongs to $T$, cf. Theorem 1.3 in \cite{plan2014high}.

    \item Finally, \citet{chatterjee2014new} exhibits connection between localized
        Gaussian widths and shape-constrained estimation.
\end{itemize}
These results are reminiscent of the isomorphic method \cite{koltchinskii2006local,bartlett2005local,bartlett2006empirical},
where localized expected supremum of empirical processes are used to obtain upper bounds on the performance of Empirical Risk Minimization (ERM) procedures.
These results show that Gaussian width estimates are important to understand
the statistical properties of estimators in many statistical contexts.

In \Cref{prop:gw-simple}, we established an upper bound on the Gaussian width of $M$-convex hulls.
We now provide some statistical applications of this result in regression with fixed-design. We will use the following Theorem from \cite{bellec2015shape}.

\begin{thm}[\cite{bellec2015shape}]
    \label{thm:isomorphic}
    Let $K$ be a closed convex subset
    of $\mathbf{R}^n$ and
    $\vxi\sim \mathcal{N}(0, \sigma^2 I_{n\times n})$.
    Let 
    $\vf_0\in\Rn$  be an unknown vector
    and let $\vy=\vf_0 + \vxi$.
    Denote by $\vf_0^*$ the projection of $\vf_0$ onto $K$.
    Assume that for 
    some $t_* >0$,
    \begin{equation}
    \label{eq:fixed-point-C}
    \frac 1 n 
    \E \left[
            \sup_{\vu\in K:\; \frac 1 n \euclidnorms{\vf_0^* - \vu}\le t_*^2 } \vxi^T \left(\vu - \vf_0^* \right)
    \right] \le \frac{t_*^2}{2}.
    \end{equation}
    Then for any $x>0$, with probability greater than $1-e^{-x}$,
    the Least Squares estimator $\hf = \argmin_{\vf\in K}\euclidnorms{\vy - \vf}$ satisfies
    \begin{equation}
        \label{eq:soi-tstar}
        \frac 1 n \euclidnorms{\hf - \vf_0}
        \le
        \frac 1 n \euclidnorms{\vf_0^* - \vf_0}
        + 2t_*^2 + \frac{4\sigma^2 x}{n}.
        \label{eq:general-oracle-ineq-K}
    \end{equation}
\end{thm}

Hence, to prove an oracle inequality of the form \eqref{eq:general-oracle-ineq-K}, it is enough to prove the existence of a quantity $t_*$ such that \eqref{eq:fixed-point-C} holds.
If the convex set $K$ in the above theorem is the convex hull of $M$ points, then a quantity $t_*$ is given by the following proposition.

\begin{prop}
    \label{prop:fixed-point-kappa}
    Let $\sigma^2 > 0, R > 0, n\ge1$ and $M\ge2$.
    Let $\vmu_1,...,\vmu_M \in \Rn$ such that $\frac 1 n |\vmu_j|_2^2 \le R^2$ for all $j=1,...,M$.
    For all $\vtheta\in\simplex$,
let $\vmu_\vtheta = \sum_{j=1,...,M} \theta_j\vmu_j$.
    Let $\vg$ be a centered Gaussian random variable with covariance matrix $\sigma^2I_{n\times n}$.
    If
    $R\sqrt n\le M\sigma$
    then the quantity
    \begin{equation}
        t_*^2 = 31 \sigma R \sqrt{\frac{\log\left(\frac{ eM \sigma}{R \sqrt n} \right)}{n}}
    \quad
    \text{satisfies}
    \quad
        \frac 1 n
    \E \sup_{\vtheta\in\simplex: \frac 1 n |\vmu_\vtheta|_2^2 \le t_*^2} \vg^T \vmu_\vtheta
    \le \frac{t_*^2}{2},
        \label{eq:gaussian-width}
    \end{equation}
    provided that $t_* \le R$.
\end{prop}

\begin{proof}
    Inequality
    \begin{equation}
        \frac{1}{\sqrt n}
        \E \sup_{\vtheta\in\simplex: \frac 1 n |\vmu_\vtheta|_2^2 \le r^2} (\sigma\vg)^T \vmu_\vtheta
        \le 4 \sigma R \sqrt{ \log\left( 4 eM \min(1, r^2/R^2 ) \right) }.
    \end{equation}
    is a reformulation of \Cref{prop:gw-simple}
    using the notation of \Cref{prop:fixed-point-kappa}.
    Thus, in order to prove \eqref{eq:gaussian-width}, it is enough to establish
that for $\gamma=31$ we have
\begin{equation}
    (*) \coloneqq 64
    \log\left( \frac{4 eM \sigma \gamma \sqrt{\log(eM\sigma/(R \sqrt n))}}{R\sqrt n} \right)
    \le
    \frac{\gamma^2}{4}
    \log\left( 
        \frac{eM \sigma}{ R \sqrt n}
    \right).
    \label{eq:enough}
\end{equation}
As $1\le\log(eM\sigma/(R\sqrt n))$ and $\log t \le t$ for all $t>0$, 
the left hand side of the previous display satisfies
\begin{align}
   (*)
   &\le
   64
    \left(
    \log\left( \frac{eM \sigma }{R\sqrt n} \right)
    +
    \log(4\gamma)
    +
    \frac 1 2 
    \log\left(
         \log(eM\sigma/(R \sqrt n))
    \right)
    \right), \\
    &\le
    64
    ( 3/2 + 
        \log(4\gamma)
    )
        \log\left( \frac{eM \sigma }{R\sqrt n} \right).
\end{align}
Thus \eqref{eq:enough} holds if
$64 (3/2+\log(4\gamma)) \le \gamma^2/4$,
which is the case if
the absolute constant is $\gamma=31$.
\end{proof}

Inequality  \eqref{eq:gaussian-width} establishes the existence of a quantity $t_*$
such that
\begin{equation}
    \frac 1 n \E \sup_{\vmu\in T: \frac 1 n |\vmu|_2^2 \le t_*^2} \vg^T \vmu_\vtheta
    \le \frac{t_*^2}{2},
    \label{eq:example-factor2}
\end{equation}
where $T$ is the convex hull of $\vmu_1,...,\vmu_M$.
Consequences of \eqref{eq:example-factor2} and \Cref{thm:isomorphic} are given in the next subsections.

We now introduce two statistical frameworks where 
the localized Gaussian width of an $M$-convex hull
has applications: the Lasso estimator in high-dimensional statistics
and the convex aggregation problem.

\subsection{Convex aggregation}
Let $\vf_0\in\R^n$ be an unknown regression vector
and let $\vy = \vf_0 + \vxi$ be an observed random vector,
where $\vxi$ satisfies $\E [ \vxi ] = \vzero$. 
Let $M\ge 2$
and let $\vf_1,...,\vf_M$ be deterministic vectors in $\R^n$.
The set $\{\vf_1,...,\vf_M\}$ will be referred to as the dictionary.
For any $\vtheta=(\theta_1,...,\theta_M)^T \in\mathbf{R}^M$, let 
$\vf_\vtheta = \sum_{j=1}^M \theta_j \vf_j$.
If a set $\Theta\subset\R^M$ is given,
the goal of the aggregation problem induced by $\Theta$
is to find an estimator $\hf$ constructed with $\vy$ and the dictionary
such that
\begin{equation}
    \label{eq:soi}
    \frac 1 n \euclidnorms{\hf - \vf_0}
    \le
    \inf_{\vtheta\in\Theta} \left( \frac 1 n \euclidnorms{\vf_\vtheta - \vf_0} \right) + \delta_{n,M,\Theta},
\end{equation}
either in expectation or with high probability, where $\delta_{n,M,\Theta}$ is a small quantity.
Inequality \eqref{eq:soi} is called a sharp oracle inequality,
where "sharp" means that in the right hand side of \eqref{eq:soi},
the multiplicative constant of the term $\inf_{\vtheta\in\Theta} \frac 1 n \euclidnorms{\vf_\vtheta - \vf_0}$
is $1$.
Similar notations will be defined for regression with random design and density estimation.
Define the simplex in $\RM$ by \eqref{eq:def-simplex}.
The following aggregation problems were introduced in
\cite{nemirovski2000lecture,tsybakov2003optimal}.
\begin{itemize}
    \item
        \emph{Model Selection type aggregation} with $\Theta = \{\ve_1,...,\ve_M\}$, i.e., $\Theta$ is the canonical basis of $\R^M$.
        The goal is to construct an estimator whose risk
        is as close as possible to the best function in the dictionary.
        Such results can be found in
        \cite{tsybakov2003optimal,lecue2014optimal,audibert2007fast}
        for random design regression,
        in \cite{leung2006information,dai2012deviation,bellec2014affine,dalalyan2007aggregation} for fixed design regression,
        and in \cite{juditsky2008learning,bellec2014optimal} for density estimation.
    \item
        \emph{Convex aggregation} with $\Theta=\simplex$, i.e., $\Theta$ is the simplex in $\R^M$.
        The goal is to construct an estimator whose risk is as
        close as possible to the best convex combination of the dictionary
        functions.
        See \cite{tsybakov2003optimal,lecue2009aggregation,lecue2013empirical,tsybakov2014aggregation}
        for results of this type in the regression framework
        and \cite{rigollet2007linear} for such results in density estimation.
    \item
        \emph{Linear aggregation} with $\Theta = \R^M$.
        The goal is to construct an estimator whose risk is as
        close as possible to the best linear combination of the dictionary
        functions,
        cf. \cite{tsybakov2003optimal,tsybakov2014aggregation}
        for such results in regression 
        and \cite{rigollet2007linear} for such results in density estimation.
\end{itemize}
One may also define the \emph{Sparse} or \emph{Sparse Convex} aggregation problems: construct an estimator whose risk is as
close as possible to the best sparse combination of the dictionary
functions. 
Such results
can be found in \cite{rigollet2012sparse,rigollet2011exponential,tsybakov2014aggregation} 
for fixed design regression
and in \cite{lounici2007generalized}
for regression with random design.
These problems are out of the scope of the present paper.

A goal of the present paper is to provide a unified
argument that shows that empirical risk minimization
is optimal for the convex aggregation problem
in density estimation, regression with fixed design
and regression with random design.

\begin{thm}
    \label{thm:fixed-convex-agg}
    Let $\vf_0\in\R^n$,
    let $\vxi\sim \mathcal{N}(0, \sigma^2 I_{n\times n})$
    and define $\vy = \vf_0+\vxi$.
    Let $\vf_1,...,\vf_M\in \R^n$ and let $\vf_\vtheta = \sum_{j=1}^M\theta_j \vf_j$
    for all $\vtheta=(\theta_1,...,\theta_M)^T\in\RM$.
    Let 
    \begin{equation}
        \htheta \in \argmin_{\vtheta\in\simplex} \euclidnorms{\vf_\vtheta - \vy}.
    \end{equation}
    Then for all $x>0$, with probability greater than $1-\exp(-x)$,
    \begin{equation}
        \frac 1 n \euclidnorms{\vf_\htheta - \vf_0}
        \le
        \min_{\vtheta\in\simplex}
        \frac 1 n \euclidnorms{\vf_\vtheta - \vf_0}
        +
        2 
        t_*^2
        + \frac{4 \sigma^2 x }{n},
    \end{equation}
    where 
        $t_*^2  =
        \min\left( 
            \frac{4\sigma^2M }{n},
            \frac{31 \sigma R \sqrt{\log(eM\sigma/(R\sqrt n))}}{\sqrt n}
        \right)
        $
        and
    $R^2 = \frac 1 4 \max_{j=1,...,M} \frac 1 n \euclidnorms{\vf_j}$.
\end{thm}
\begin{proof}[Proof of \Cref{thm:fixed-convex-agg}] 
    Let $V$ be the linear span of $\vf_1,...,\vf_M$
    and let $P\in\R^{n\times n}$ be the orthogonal projector onto $V$.
    If $t_*^2 = 4 \sigma^2 M / n$, then
    \begin{equation}
        \frac 1 n \E \sup_{\vv\in V : \frac 1 n \euclidnorms{\vv}\le t_*^2} \vxi^T \vv
        =
        \sqrt{ \frac{t_*^2}{n} } \E \euclidnorm{P\vxi}
        \le 
        \sqrt{ \frac{t_*^2}{n} } \sqrt{ \E \euclidnorms{P\vxi} }
        = 
        \sqrt{ \frac {t_*^2 \sigma^2 M} n }
        = t_*^2 / 2.
    \end{equation}
    Let $K$ be 
    the convex hull of $\vf_1,...,\vf_M$.
    Let $\vf_0^*$ be the convex projection of $\vf_0$
    onto $K$.
    We apply \Cref{prop:fixed-point-kappa}
    to $K - \vf_0^*$ which is a convex hull of $M$ points,
    and for all $\vv\in K$, $\frac 1 n \euclidnorms{\vv} \le R^2$.
    By \eqref{eq:standard-Mn-design} and \eqref{eq:gaussian-width},
    the quantity $t_*$
    satisfies \eqref{eq:fixed-point-C}.
    Applying \Cref{thm:isomorphic} completes the proof.
\end{proof}

\subsection{Lasso}
We consider the following regression model.
Let $\vx_1,...,\vx_M\in R^n$ and assume that $\frac 1 n \euclidnorms{\vx_j}\le 1$ for all $j=1,...,M$.
We will refer to $\vx_1,...,\vx_M$ as the covariates.
Let $\design$ be the matrix of dimension $n\times M$
with columns $\vx_1,...,\vx_M$.
We observe
\begin{equation}
    \vy = \vf_0+\vxi,
    \qquad
    \vxi\sim \mathcal{N}(0, \sigma^2 I_{n\times n}).
    \label{eq:linear-reg-model}
\end{equation}
where $\vf_0\in\R^n$ is an unknown mean.
The goal is to estimate $\vf_0$ using the
design matrix $\design$.

Let $R>0$ be a tuning parameter and
define the constrained Lasso estimator
\cite{tibshirani1996regression} by
\begin{equation}
    \hbeta \in\argmin_{\vbeta\in\R^M: \onenorm{\vbeta}\le R}
    \euclidnorms{\vy - \design \vbeta}.
    \label{eq:def-lasso}
\end{equation}
Our goal will be to study the performance
of the estimator \eqref{eq:def-lasso} with respect
to the prediction loss
\begin{equation}
    \frac 1 n \euclidnorms{\vf_0 - \design\hbeta}.
\end{equation}

\label{s:lasso}
Let $\vx_1,...,\vx_M\in R^n$ and assume that $\frac 1 n \euclidnorms{\vx_j}\le 1$ for all $j=1,...,M$.
Let $\design$ be the matrix of dimension $n\times M$
with columns $\vx_1,...,\vx_M$.
\begin{thm}
    \label{thm:lasso}
    Let $R>0$ be a tuning parameter
    and consider the regression model
    \eqref{eq:linear-reg-model}.
    Define the Lasso estimator $\hbeta$ by
    \eqref{eq:def-lasso}.
    Then for all $x>0$, with probability greater than $1-\exp(-x)$,
    \begin{equation}
        \label{eq:soilasso}
        \frac 1 n \euclidnorms{\design\hbeta - \vf_0}
        \le
        \min_{\vbeta\in\R^M: \onenorm{\vbeta}\le R}
        \frac 1 n \euclidnorms{\design\vbeta - \vf_0}
        +
        2  t_*^2
        + \frac{4 \sigma^2 x }{n},
    \end{equation}
    where 
        $t_*^2  =
        \min\left( 
            \frac{4\sigma^2\rank(\design)}{n},
            \frac{62 \sigma R \sqrt{\log(2eM\sigma/(R\sqrt n))}}{\sqrt n}
        \right)
        $.
\end{thm}
\begin{proof}[Proof of \Cref{thm:lasso}] 
    Let $V$ be the linear span of $\vx_1,...,\vx_M$
    and let $P\in\R^{n\times n}$ be the orthogonal projector onto $V$.
    If $ t_*^2 = 4 \sigma^2 \rank(\design) / n$, then
    \begin{equation}
        \label{eq:standard-Mn-design}
        \frac 1 n \E \sup_{\vv\in V : \frac 1 n \euclidnorms{\vv}\le t_*^2} \vxi^T \vv
        =
        \sqrt{ \frac{t_*^2}{n} } \E \euclidnorm{P\vxi}
        \le 
        \sqrt{ \frac{t_*^2}{n} } \sqrt{ \E \euclidnorms{P\vxi} }
        = 
        \sqrt{ \frac {t_*^2 \sigma^2 \rank(\design)} n }
        = t_*^2 / 2.
    \end{equation}
    Let $K$ be 
    the convex hull of $\{\pm R\vx_1,...,\pm R\vx_M\}$,
    so that $K = \{\design\vbeta: \vbeta\in\R^M: \onenorm{\vbeta}\le R\}$.
    Let $\vf_0^*$ be the convex projection of $\vf_0$
    onto $K$.
    We apply \Cref{prop:fixed-point-kappa}
    to $K - \vf_0^*$ which is a convex hull of $2M$ points of empirical norm less or equal to $R^2$.
    By \eqref{eq:standard-Mn-design} and \eqref{eq:gaussian-width},
    the quantity $t_*$
    satisfies \eqref{eq:fixed-point-C}.
    Applying \Cref{thm:isomorphic} completes the proof.
\end{proof}
The lower bound \cite[Theorem 5.4 and (5.25)]{rigollet2011exponential} states that
there exists an absolute constant $C_0>0$ such that the following holds.
If $\log(1+eM/\sqrt n)\le C_0\sqrt n$, then
there exists a design matrix $\design$ such that
for all estimator $\hf$,
\begin{equation}
    \sup_{\vbeta\in\R^M: \onenorm{\vbeta}\le R}
    \frac 1 n
    \E_{\design\vbeta}
            \euclidnorms{\design\vbeta - \hf}
    \ge \frac 1 {C_0}
        \min\left(
            \frac{\sigma^2 \rank(\design)}{n}
            ,
            \sigma R \sqrt{
                \frac{\log\left( 1 + \frac{eM\sigma}{R\sqrt n} \right)}{n}
            }
        \right),
\end{equation}
where for all $\vf_0\in\R^n$, 
$\E_{\vf_0}$ denotes the expectation with respect
to the distribution of $\vy\sim\mathcal N(\vf_0,\sigma^2I_{n\times n})$.
Thus, \Cref{thm:lasso} shows that the Least Squares estimator
over the set $\{\design\vbeta, \vbeta\in\R^M: \onenorm{\vbeta} \le R\}$
is minimax optimal.
In particular, the right hand side of inequality \eqref{eq:soilasso}
cannot be improved.

\section{The anisotropic persistence problem in regression with random design}
\label{s:persistence}

Consider $n$ iid observations $(Y_i,X_i)_{i=1,...,n}$
where $(Y_i)_{i=1,...,n}$ are real valued and the $(X_i)_{i=1,..,n}$ are design random variables in $\R^M$ with $\E[X_iX_i^T] = \Sigma$
for some covariance matrix $\Sigma\in\R^{M\times M}$.
We consider the learning problem over the function class
\begin{equation}
    \left\{
        f_\beta:\quad
        f_\beta(x) = x^T\beta
        \text{ for some }
        \beta\in\R^M
        \text{ with } |\beta|_1 \le R
    \right\}
\end{equation}
for a given constant $R>0$.
We consider the Emprical Risk Minimizer defined by
\begin{equation}
    \hbeta = \argmin_{\vbeta\in\R^M: |\vbeta|_1\le R} \sum_{i=1}^n (Y_i - \vbeta^TX_i)^2
    \label{eq:persistence-def-hbeta}
\end{equation}
This problem is sometimes referred to as the persistence problem
or the persistence framework
\cite{greenshtein2004persistence,bartlett2012l1}.
The prediction risk of $f_\hbeta$ is given by
\begin{equation}
    R(f_\hbeta) = \E[ (f_\hbeta(X) - Y) \; | \; (X_i,Y_i)_{i=1,...,n} ],
\end{equation}
where $(X,Y)$ is a new observation distributed as $(X_1,Y_1)$ and independent from the data $(X_i,Y_i)_{i=1,...,n}$.
Define also the oracle $\vbeta^*$ by
\begin{equation}
    \vbeta^* = \argmin_{\vbeta\in\R^M: |\vbeta|_1\le R} R(\vbeta)
    \label{eq:persistence-def-oracle}
\end{equation}
and define $\sigma >0$ by
\begin{equation}
    \sigma = \|Y - X^T\vbeta^*\|_{\psi_2},
    \label{eq:persistence-def-sigma}
\end{equation}
where the subgaussian norm $\|\cdot\|_{\psi_2}$ is defined by
$\|Z\|_{\psi_2} = \sup_{p\ge 1} \E[|Z|^p]^{1/p}/\sqrt p$
for any random variable $Z$ (see Section 5.2.3 in \cite{vershynin2010introduction} for equivalent definitions of the $\psi_2$ norm).

To analyse the above learning problem, we use the machinery developed by \citet{lecue2013learning} to study learning problems over subgaussian classes.
Consider the two quantities
\begin{align}
    r_n(\gamma) &= \inf\left\{
        r>0:
        \E
        \sup_{\vbeta: |\vbeta|_1\le 2R,\; \E[(G^T\vbeta)^2]\le s^2}
        \vbeta^T G
        \le \gamma r \sqrt n
    \right\}, 
    \label{eq:def-r-n}
    \\
    s_n(\gamma) &= \inf\left\{
        s>0:
        \E
        \sup_{\vbeta: \;|\vbeta|_1\le 2R,\; \E[(G^T\vbeta)^2]\le s^2}
        \vbeta^T G
        \le \gamma s^2 \sqrt n / \sigma
    \right\},
    \label{eq:def-s-n}
\end{align}
where $G\sim N(\vzero,\Sigma)$.
In the present setting, Theorem A from \citet{lecue2013learning} reads as follows.
\begin{thm}[Theorem A in \citet{lecue2013learning}]
    \label{cor:anisotropy}
    There exist absolute constants $c_1,c_2,c_4>0$ such that the following holds.
    Let $R>0$.
    Consider iid observations $(X_i,Y_i)$ with $\E[X_iX_i^T]=\Sigma$. 
    Assume that the design random vectors $X_i$ are subgaussian with respect to the covariance matrix $\Sigma$ in the sense that
    $\|X_i^T\tau\|_{\psi_2} \le 10 |\Sigma^{1/2}\tau|_2$ for any $\tau\in\R^p$.
    Define $\vbeta^*$ by \eqref{eq:persistence-def-oracle} and
    $\sigma$ by \eqref{eq:persistence-def-sigma}.
    Assume that the diagonal elements of $\Sigma$ are no larger than 1.
    Then, there exists absolute constants $c_0,c_1,c_2,c_3>0$ such that  the estimator $\hbeta$ defined in 
    \eqref{eq:persistence-def-hbeta}
    satisfies
    \begin{equation}
        R(f_\hbeta) \le R(f_{\vbeta^*}) + \max(s_n^2(c_1), r_n^2(c_2)),
    \end{equation}
    with probability at least $1-6\exp(-c_4 n \min(c_2,s_n(c_1)))$.
\end{thm}

In the isotropic case ($\Sigma = I_M$),
\cite{mendelson2014learning} proves that
\begin{equation}
    \label{eq:upper-b}
r_n^2(\gamma)\le 
\begin{cases}
    \frac{c_3R^2}{n} \log(\frac{c_3 M}{n}) &\text{if } n \le c_4 M\\
    0  &\text{if } n> c_4 M,
\end{cases}
\end{equation}
for some constants $c_3,c_4>0$ that only depends on $\gamma$,
while
\begin{equation}
    \label{eq:upper-a}
s_n^2(\gamma)\le 
\begin{cases}
    \frac{c_5R\sigma}{\sqrt n} \sqrt{\log(\frac{c_5 M \sigma }{\sqrt n R})} &\text{if } n \le c_6 \sigma^2 M^2 / R^2\\
    \frac{c_5 \sigma^2 M}{n}   &\text{if } n> c_6 \sigma^2 M^2 / R^2,
\end{cases}
\end{equation}
for some constants $c_5,c_6>0$ that only depend on $\gamma$.

Using \Cref{prop:gw-simple} and \Cref{eq:def-that-maurey} above lets us
extend these bounds to the anisotropic case where $\Sigma$ is not proportional to the identity matrix.

\begin{prop}
    \label{prop:anisotropy}
    Let $R>0$,
    let $G \sim N(\vzero,\Sigma)$ and assume that the diagonal elements of $\Sigma$ are no larger than 1.
    For any $\gamma>0$, define $r_n(\gamma)$ and $s_n(\gamma)$
    by \eqref{eq:def-s-n} and \eqref{eq:def-r-n}.
    Then for any $\gamma>0$, there exists constants $c_3,c_4,c_5,c_6>0$ that depend only on $\gamma$ such that
    \eqref{eq:upper-a} and \eqref{eq:upper-b} hold.
\end{prop}
The proof of \Cref{prop:anisotropy} will be given at the end of this subsection.
The primary improvement of \Cref{prop:gw-simple} over previous results is that this result is agnostic to the underlying covariance structure.
This lets us handle the anisotropic case with $\Sigma\ne I_M$ in the above proposition.

\Cref{prop:anisotropy} combined with \Cref{cor:anisotropy} lets us obtained the minimax rate of estimation for the persistence problem
in the anisotropic case. Although the minimax rate was previously obtained in the isotropic case,
we are not aware of a previous result that yields this rate for general covariance matrices $\Sigma\ne I_M$.

\begin{proof}[Proof of \Cref{prop:anisotropy}]
    In this proof, $c>0$ is an absolute constant whose value may change from line to line.
    Let $\gamma>0$. We first bound $r_n(\gamma)$ from above.
    Let $r>0$ and define
    \begin{equation}
        T_r(R) = \{ \vbeta \in \R^p: |\vbeta|_1\le 2R, \vbeta^T\Sigma\vbeta \le r^2 \}.
    \end{equation}
    The random variable $X\sim N(\vzero,\Sigma)$ has the same distribution
    as $\Sigma^{1/2}\vg$ where $\vg\sim N(\vzero,I_M)$.
    Thus, the expectation inside the infimum in \eqref{eq:def-r-n} is equal to
    \begin{equation}
        \E
        \sup_{\vbeta\in T_r(R)} \vbeta^T \Sigma^{1/2}\vg.
        \label{eqijofeq}
    \end{equation}
    To bound $r_n(\gamma)$ from above, it is enough to find some $r>0$ such that
    \eqref{eqijofeq} is bounded from above by $\gamma r \sqrt n$.

    By the Cauchy-Schwarz inequality, the right hand side is bounded from
    above by $r \sqrt M$, which is smaller than $\gamma r \sqrt n$
    for all small enough $s>0$ provided that $n>c_4 M$
    for some constant $c_4$ that only depends on $\gamma$.

    We now bound $r_n(\gamma)$ from above in the regime $n \le c_4 M$.
    Let $\vu_1,...,\vu_M$ be the columns of $\Sigma$
    and let $\tilde T$ be the convex hull of the $2M$ points $\{\pm \vu_1, ..., \pm\vu_M\}$.
    Using the fact that $T_r(R) = 2R T_{r/(2R)}(1) \subset 2R(\tilde T\cap (r/(2R))B_2)$,
    the right hand side of the previous display is bounded from above by
    \begin{equation}
        2 R \; \ell(\tilde T\cap(r/R))B_2)
        \le
        8 R 
        \sqrt{\log_+(4eM(\tfrac{r}{2R})^2)},
        \label{eqlfewji}
    \end{equation}
    where we used \Cref{prop:gw-simple} for the last inequality.
    By simple algebra, one can show that if $r = c_3(\gamma) \frac{R}{\sqrt n}\sqrt{\log(c_3(\gamma) M /n )}$
    for some large enough constant $c_3(\gamma)$ that only depends on $\gamma$,
    then
    the right hand side of \eqref{eqlfewji}
    is bounded from above by $\gamma r \sqrt n$.

    We now bound $s_n(\gamma)$ from above. 
    Let $s>0$. By definition of $s_n(\gamma)$,
    to prove that $s_n(\gamma) \le s$, it is enough to show that
    \begin{equation}
    \sigma
        \E_\xi
        \sup_{\vbeta\in T_s(R)} \vbeta^T \Sigma^{1/2}\vg
    \end{equation}
    is smaller than $\gamma s^2 \sqrt n$.
    We use \Cref{prop:gw-simple} to show that the right hand side of the previous display is bounded from above by
    \begin{equation}
        c\sigma \min\left(
        s \sqrt M,
        R \sqrt{\log_+(4eM(\tfrac{s}{2R})^2)}
        \right).
    \end{equation}
    By simple algebra very similar to that of the proof of \Cref{prop:fixed-point-kappa},
    we obtain that if $s^2$ equals the right hand side of \eqref{eq:upper-a}
    for large enough $c_5=c_5(\gamma)$ and $c_6=c_6(\gamma)$,
    then the right hand side of the previous display is bounded from above
    by $\gamma s^2 \sqrt n$. This completes the proof of \eqref{eq:upper-a}.
\end{proof}


\section{Bounded empirical processes and density estimation \label{s:bounded}}

We now prove a result similar to \Cref{prop:gw-simple} for bounded empirical processes indexed by the convex hull of $M$ points.
This will be useful to study the convex aggregation problem for density
estimation.
Throughout the paper, $\epsilon_1,...,\epsilon_n$ are i.i.d. Rademacher random variables that are independent of all
other random variables.

\begin{prop}
    \label{prop:lambdastar2}
    There exists an absolute constant $c > 0$ such that the following holds.
    Let $M\ge2,n\ge1$ be integers and let $b_{},R,L>0$ be real numbers.
    Let $Q(\vtheta) = \vtheta^T\Sigma\vtheta$ for some semi-positive matrix $\Sigma$.
    Let $Z_1,...,Z_n$ be i.i.d. random variables valued in some measurable set $\mathcal Z$.
    Let $h_1,...,h_M:\mathcal Z \rightarrow \R$ be measurable functions.
    Let $h_\vtheta = \sum_{j=1}^M\theta_j h_j$ for all $\vtheta=(\theta_1,...,\theta_M)^T\in\RM$.
    Assume that almost surely
    \begin{equation}
        |h_j(Z_1)|\le b_{},
        \qquad
        Q(\ve_j) = \Sigma_{jj} \le R^2,
        \qquad
        \E [h_\vtheta^2(Z_1)] \le L Q(\vtheta),
        \label{eq:assumptions-bounded-Q-b-R}
    \end{equation}
    for all $j=1,...,M$ and all $\vtheta\in\simplex$.
    Then for all $r>0$ such that $R/\sqrt M \le r \le R$ we have
    \begin{equation}
        \E\left[ \sup_{\vtheta\in\simplex:\; Q(\vtheta) \le r^2} F(\vtheta) \right]
        \le
        c  \max\left(
            \sqrt L R  \sqrt{\frac{ \log(eMr^2/R^2)}{n}}
            ,
            \frac{b_{} R^2 \log(eMr^2/R^2)}{r^2n}
        \right),
        \label{eq:rhs-lambdastar2}
    \end{equation}
    where $F(\vtheta) = \frac 1 n |\sum_{i=1}^n \epsilon_i h_\vtheta(Z_i)|$
    for all $\vtheta\in\RM$.
\end{prop}

\begin{proof}[Proof of \Cref{prop:lambdastar2}]
Let $m = \lfloor R^2/r^2 \rfloor \ge 1$.
The function $F$ is convex since it can be written as the maximum of two linear functions.
Applying \Cref{lemma:maureysup} with the above notation
and $t = r$
yields
\begin{equation}
    \E \sup_{\vtheta\in\simplex: Q(\vtheta)\le r^2} 
    F(\vtheta)
        \le
        \E
        \int_1^{+\infty}
        M(x)
        \frac{dx}{x^2}
        =
        \int_1^{+\infty}
        \E
        \left[
            M(x)
        \right] \frac{dx}{x^2}.
    \label{eq:afewA}
\end{equation}
where the second inequality is a consequence of Fubini's Theorem and for all $x\ge1$,
$$
M(x) = 
            \max_{\vtheta\in\simplex_m:\; Q(\vtheta) \le x(r^2 + R^2/m)}
            F(\vtheta)
.
$$
Using \eqref{eq:assumptions-bounded-Q-b-R} and the Rademacher complexity bound for finite classes
given in \cite[Theorem 3.5]{koltchinskii2011oracle},
we obtain that for all $x\ge1$,
\begin{align}
    \E [M(x)]
    \le c' \max\left(
        \sqrt{  \frac{L x(r^2+R^2/m)\log |\simplex_m|}{n}}
    ,
    \frac{b_{} \log|\simplex_m|}{n}
    \right),
    \label{eq:afewB}
\end{align}
where $c'>0$ is a numerical constant
and $|\simplex_m|$ is the cardinal of the set $\simplex_m$.
By definition of $m$ we have $r^2\le R^2/m$.
The cardinal $|\simplex_m|$ of the set $\simplex_m$
is bounded from above by the right hand side of 
\eqref{eq:bound-cardinal}.
Combining inequality \eqref{eq:afewA},
inequality \eqref{eq:afewB},
the fact that the integrals 
$
        \int_1^{+\infty}
        \frac{dx}{x^2}$
and
$
        \int_1^{+\infty}
        \frac{dx}{x^{3/2}}$
are finite, we obtain
\begin{equation}
    \E \sup_{\vtheta\in\simplex: Q(\vtheta)\le r^2} 
    F(\vtheta)
        \le
        c'' \max\left(
            \sqrt L  R \sqrt{\frac{\log(eM/m)}{n}},
            \frac{b_{} m \log(eM/m)}{n}
        \right)
\end{equation}
for some absolute constant $c''>0$.
By definition of $m$, we have $R^2/(2r^2)\le m\le R^2/r^2$.
A monotonicity argument completes the proof.
\end{proof}

Next, we show that
\Cref{prop:lambdastar2}
can be used to derive a condition
similar to
\eqref{eq:gaussian-width} for bounded empirical processes.
To bound from above the performance of ERM procedures in density estimation,
\Cref{thm:iso-density} in the appendix requires the existence of a quantity $r_*>0$
such that
\begin{equation}
    \E\left[ \sup_{\vtheta\in\simplex:\; Q(\vtheta) \le r_*^2} F(\vtheta) \right]
    \le \frac{r_*^2}{16},
    \label{eq:fixedpoint-KRC}
\end{equation}
where $F$ is the function defined in \Cref{prop:lambdastar2} above.

To obtain such quantity $r_*>0$ under the assumptions of \Cref{prop:lambdastar2},
we proceed as follows.
Let $K = \max(b_{},\sqrt{L})$
and
assume that
\begin{equation}
    M K > R \sqrt n.
\end{equation}
Define 
$r^2 = C K R  \sqrt{\log(eM K / (R\sqrt n))}$
where $C\ge1$ is a numerical constant that will be chosen later.
We now bound from above the right hand side of
\eqref{eq:rhs-lambdastar2}.
We have
\begin{align}
    \log(eMr^2/R^2)
    &\le
    \log(C) + \log(eM K/(R \sqrt n)) + (1/2) \log\log(eM K/(R \sqrt n)) 
    \\
    &\le
    (\log(C) + 3/2) \log(eM K/(R \sqrt n)),
    \label{eq:bound-log+3/2}
\end{align}
where for the last inequality we used that $\log\log(u)\le\log u$
for all $u>1$ and that
$\log(C) \le \log(C) \log(eM K/(R \sqrt n))$,
since $C\ge 1$ and $MK/(R\sqrt n)\ge 1$.
Thus, 
the right hand side of 
\eqref{eq:rhs-lambdastar2} is bounded from above
by
\begin{equation}
    c
    \max\left(
        \frac{\sqrt{\log(C)+3/2}}{C},
        \frac{\log(C)+3/2}{C^2}
    \right)
    r^2.
\end{equation}
It is clear that the above quantity is bounded from above by $r^2/16$ if
the numerical constant $C$ is large enough.
Thus we have proved that 
as long as $M K > R \sqrt n$,
inequality \eqref{eq:fixedpoint-KRC} holds for
\begin{equation}
    r_*^2 = C R K \sqrt { \frac{\log(eMK/(R\sqrt n))}{n} },
\end{equation}
where $C\ge1$ is a numerical constant.

\subsection*{ERM and convex aggregation in density estimation \label{s:convex}}
The minimax optimal rate for the convex aggregation problem is known
to be of order
\begin{equation}
    \label{eq:phiconvex}
    \phi_M^C(n) \coloneqq  \min\left( \frac{M}{n},\sqrt{\frac{\log\left(\frac{eM}{\sqrt{n}}\right)}{n}} \right)
\end{equation}
for regression with fixed design 
\cite{rigollet2011exponential}
and regression with random design
\cite{tsybakov2003optimal}
if the integers $M$ and $\sqrt n$ satisfy
$eM\sigma\le R\sqrt n \exp(\sqrt n)$ or equivalently $\phi_M^C(n)\le 1$.
The arguments for the convex aggregation lower bound
from \cite{tsybakov2003optimal}
can be readily applied to density estimation,
showing that the rate
$\phi_M^C(n)$ is a lower bound on the optimal rate of convex aggregation for density estimation. 

We now use the results of the previous sections to show that ERM
is optimal for the convex aggregation problem
in regression with fixed design, regression with random design and density estimation.

\begin{thm}
    There exists an absolute constant $c>0$ such that the following holds.
    Let $(\mathcal Z,\mu)$ be a measurable space with measure $\mu$.
    Let $p_0$ be an unknown density with respect to the measure $\mu$.
    Let $Z_1, ..., Z_n$ be i.i.d. random variables valued
    in $\mathcal Z$ with density $p_0$.
    Let $p_1,...,p_M\in L_2(\mu)$ and let $p_\vtheta = \sum_{j=1}^M\theta_j p_j$
    for all $\vtheta=(\theta_1,...,\theta_M)^T\in\RM$.
    Let 
    \begin{equation}
        \htheta \in \argmin_{\vtheta\in\simplex} \left( \int p_\vtheta^2 d\mu - \frac{2}{n} \sum_{i=1}^n  p_\vtheta(Z_i) \right).
    \end{equation}
    Then for all $x>0$, with probability greater than $1-\exp(-x)$,
    \begin{equation}
        \int (p_\htheta - p_0)^2d\mu
        \le
        \min_{\vtheta\in\simplex}\int (p_\vtheta - p_0)^2d\mu
        +
        c \max\left(
            \frac{b_\infty M}{n},
            R \sqrt {b_\infty}
            \sqrt{
                \frac{\log(\frac{eM \sqrt{b_\infty}}{R\sqrt n})}{n}
            }
        \right)
        + \frac{88 b_\infty x}{3n}
        ,
    \end{equation}
    where 
    $R^2 = \frac 1 4 \max_{j=1,...,M}\int p_j^2 d\mu$ and
    $b_\infty = \max_{j=0,1,...,M}\musupnorm{p_j}$.
\end{thm}
\begin{proof}
    It is a direct application of \Cref{thm:iso-density} in the appendix.
    If $M \sqrt{b_\infty}\le R\sqrt n$,
    a fixed point $t_*$ is given by \Cref{lemma:finite-d-density}.
    If $M \sqrt{b_\infty} > R\sqrt n$,
    we use 
    \Cref{prop:lambdastar2} with $Q(\vtheta) = \int (p_0^* - p_\vtheta)^2$,
    $L=b_\infty$ and $b=b_\infty$.
    The bound 
    \eqref{eq:fixedpoint-KRC}
    yields the existence of a fixed point $t_*$ in this regime.
\end{proof}

\bibliographystyle{plainnat}
\bibliography{db}

\appendix

\section{Proof of the lower bound \eqref{eq:lower-bound-RIP}}
\label{s:proof-lower-RIP}

\begin{proof}[Proof of \Cref{prop:lower-bound-RIP}]
    By the Varshamov-Gilbert extraction lemma \cite[Lemma 2.5]{giraud2014introduction}, there exist a subset $\Omega$ of $\{0,1\}^M$ such that
    \begin{equation}
        |\vomega|_0 = m,
        \qquad
        |\vomega-\vomega'|_0 > m,
        \qquad
        \log |\Omega| \ge (m/2) \log(M/(5m))
    \end{equation}
    for any distinct $\vomega,\vomega'\in\Omega$.

    For each $\vomega\in\Omega$, we define $s(\vomega)\in\{-1,0,1\}^M$, a signed version of $\vomega$, as follows.
    Let $\epsilon_1,...,\epsilon_M$ be $M$ iid Rademacher random variables.
    Then we have
    \begin{equation}
        \E[|\sum_{j=1}^M \omega_j \epsilon_j \vmu_j|_2^2]
        = \sum_{j=1}^M \omega_j |\vmu_j|_2^2 = m.
    \end{equation}
    Hence, there exists some $s(\vomega)\in\{-1,0,1\}^M$ with $|s(\vomega)_j| = \omega_j$ for all $j=1,...,M$ such that $|\vmu_{s(\vomega)}|_2^2 \le m$.

    Define $T_\Omega = \{s^2 \vmu_{s(\vomega)}, \vomega\in\Omega\}$.
Since $s^2=1/m$, each element of $T_\Omega$ is of the form
$(1/m)(\pm\vmu_{j_1}\pm...\pm\vmu_{j_m})$ where $\vmu_{j_1},...,\vmu_{j_m}$
are $m$ distinct elements of $\{\vmu_1,...,\vmu_M\}$,
hence by convexity of $T$ we have $T_\Omega\subset T$.
By definition of $s(\vomega)$, it holds that $T_\Omega\subset s B_2$,
and thus $T_\Omega\subset T \cap sB_2$.
For any two distinct $\vu,\vv\in T_\Omega$,
\begin{equation}
    |\vu- \vv|_2^2 \ge \kappa^2 s^4 \sup_{\vomega,\vomega'} |s(\vomega) - s(\vomega')|_2^2
    > \kappa^2 s^4m = \kappa^2 s^2,
\end{equation}
where the supremum is taken over any two distinct elements of $\Omega$.
By Sudakov's inequality (see for instance \cite[Theorem 13.4]{boucheron2013concentration}) we have
\begin{equation}
    \ell(T\cap sB_2)
    \ge
    \ell(T_\Omega)
    \ge
    (1/2) \kappa s \sqrt{\log \Omega}
    \ge 
    1/(2\sqrt 2)
    \kappa s \sqrt m \sqrt{\log(M/5m)}.
\end{equation}
Since $1/m=s^2$, the right hand side of the previous display is equal to the right hand side of \eqref{eq:lower-bound-RIP} and the proof is complete.
\end{proof}

\section{Local Rademacher complexities and density estimation}
\label{s:iso}

In the last decade emerged a vast literature
on local Rademacher complexities
to study the performance of empirical risk minimizers (ERM)
for general learning problems, cf. \cite{bartlett2005local,bartlett2006empirical,koltchinskii2006local}
and the references therein.
The following result is given in \cite[Theorem 2.1]{bartlett2005local}.
Let $\epsilon_1,...,\epsilon_n$ be independent Rademacher random variables,
that are independent from all other random variables considered in the paper.
\begin{thm}[\citet{bartlett2005local}]
    \label{thm:bartlett}
    Let $Z_1,...,Z_n$ be i.i.d. random variables valued in some measurable space $\mathcal Z$.
    Let $\mathcal H:\mathcal Z \rightarrow [-b_\infty,b_\infty]$ be a class of measurable functions.
    Assume that there is some $v>0$ such that $\E [h(Z_1)^2] \le v$ for all $h\in \mathcal H$.
    Then for all $x>0$, with probability greater than $1-\exp(-x)$,
    \begin{equation}
        \sup_{h\in\mathcal H}
        (P -P_n)h
        \le
        4 \E\left[ \sup_{h\in \mathcal H} \frac 1 n \sum_{i=1}^n \epsilon_i h(Z_i) \right]
        + \sqrt{\frac{2vx}{n}}
        +  \frac{8b_\infty x}{3n}.
    \end{equation}
\end{thm}
\Cref{thm:bartlett} is a straightforward consequence
of Talagrand inequality.
We now explain how \Cref{thm:bartlett} can be used to derive sharp oracle inequalities
in density estimation.

\begin{thm}
    \label{thm:iso-density}
    Let $(\mathcal Z,\mu)$ be a measurable space with measure $\mu$.
    Let $p_0$ be an unknown density with respect to the measure $\mu$.
    Let $Z_1, ..., Z_n$ be i.i.d. random variables valued
    in $\mathcal Z$ with density $p_0$.
    Let $\mathcal P$ be a convex subset of $L^2(\mu)$.
    Assume that there exists $p_0^*\in\mathcal P$ such that
    $\int (p_0-p_0^*)^2d\mu = \inf_{p\in\mathcal P} \int (p_0-p)^2d\mu$.
    Assume that for some $t_*>0$,
    \begin{equation}
        \label{eq:lambdastar-density}
        \E\left[  
                \sup_{p\in\mathcal P:\; \int (p-p_0^*)^2d\mu \le t_*^2}
                \frac 1 n \sum_{i=1}^n \epsilon_i (p-p_0^*)(Z_i)
        \right]
        \le \frac{t_*^2}{16}.
    \end{equation}
    Assume that there exists an estimator $\hat p$ such that almost surely,
    \begin{equation}
        \hat p \in \argmin_{p\in\mathcal P} \left( \int p^2 d\mu - \frac{1}{n} \sum_{i=1}^n  2 p(Z_i) \right).
    \end{equation}
    Then for all $x>0$, with probability greater than $1-\exp(-x)$,
    \begin{equation}
        \int (\hat p - p_0)^2d\mu
        \le
        \min_{p\in\mathcal P}\int (p - p_0)^2d\mu
        + 
        2 \max\left(t_*^2, \frac{4(\musupnorm{p_0}+8b_\infty/3)x}{n}\right),
    \end{equation}
    where $b_\infty = \sup_{p\in\mathcal P}\musupnorm{p}$.
\end{thm}
\begin{proof}[Proof of \Cref{thm:iso-density}]
    By optimality of $\hat p$ we have
    \begin{equation}
        \int (\hat p - p_0)^2d\mu \le
        \int (p_0^* - p_0)^2d\mu + 2\Xi_{\hat p}.
    \end{equation}
    where for all $p\in\mathcal P$, $\Xi_p$ is the random variable
    \begin{equation}
        \Xi_p = (P-P_n)(p - p_0^*) - \frac 1 2 \int (p_0^* - p)^2 d\mu.
    \end{equation}
    Let $\rho = \max\left(t_*^2, 4(\musupnorm{p_0}+8b_\infty/3)x/n\right)$ and define
    \begin{equation}
        \mathcal H = \left\{h = p_0^* - p \text{ for some } p\in\mathcal P \text{ such that } \int h^2d\mu \le \rho \right\}.
    \end{equation}
    The class $\mathcal H$ is convex, $0\in\mathcal H$ and $t_*^2\le\rho$
    so that $h\in\mathcal H$ implies $\frac{t_*^2}{\rho} h\in \mathcal H$.
    For any linear form $L$,
    \begin{equation}
        \frac 1 \rho
        \sup_{h\in \mathcal H: \int h^2 d\mu \le \rho} L(h),
        \le
        \frac 1 {t_*^2}
        \sup_{h\in \mathcal H: \int h^2 d\mu \le \rho} L\left(\frac{t_*^2}{\rho} h \right)
        \le
        \frac 1 {t_*^2} \sup_{h\in \mathcal H: \int h^2 d\mu \le t_*^2} L(h)
    \end{equation}
    so that by taking expectations, \eqref{eq:lambdastar-density} holds if $t_*^2$ is replaced by $\rho$.

    For any $h\in\mathcal H$, $\E[h(Z_1)^2] \le \musupnorm{p_0} \rho$
    and $h$ is valued in $[-2b_\infty, 2b_\infty]$ $\mu$-almost surely.
    We apply \Cref{thm:bartlett} to the class $\mathcal H$.
    This yields that with probability greater than $1-e^{-x}$,
    if $p\in\mathcal P$ is such that $p_0^8-p\in\mathcal H$, then
    \begin{align}
        (P-P_n)(p_0^* - p)
        &\le \frac{\rho}{4}
        + 
        \sqrt \frac{2\rho\musupnorm{p_0}x}{n}
        + \frac{16b_\infty x}{n}, \\
        &\le
        \frac \rho 2
        + 2\left( \musupnorm{p_0} + 8b_\infty/3 \right)\frac x n
        \le \rho
        .
    \end{align}
    On the same event of probability greater than $1-e^{-x}$,
    if $p\in\mathcal P$ is such that $\int(p_0^* - p)^2d\mu > \rho$,
    consider $h =\sqrt \rho (p_0^* - p) / \sqrt{\int(p_0^*-p)^2d\mu}$ which belongs to $\mathcal H$.
    We have $(P-P_n)h \le \rho$, which can be rewritten 
    \begin{equation}
        (P-P_n)(p_0^* - p) \le \sqrt \rho \sqrt{\int(p_0^* - p)^2d\mu} \le \rho/2 + \int(p_0^* - p)d\mu / 2,
    \end{equation}
    so that $\Xi_p \le \rho/2 \le \rho$.
    In summary, we have proved that on an event of probability greater than $1-e^{-x}$,
    $\sup_{p\in\mathcal P} \Xi_p \le \rho$.
    In particular, this holds for $p=\hat p$ which completes the proof.
\end{proof}

\section{A fixed point $t_*$ for finite dimensional classes}

\begin{lemma}
    \label{lemma:finite-d-density}
    Consider the notations of \Cref{thm:iso-density}
    and assume that the linear span of $\mathcal P$
    is finite dimensional of dimension $d$.
    Then \eqref{eq:lambdastar-density}
    is satisfied for $t_*^2 = 256 \musupnorm{p_0} d/n$.
\end{lemma}

\begin{proof}
Let $e_1, ..., e_{d}$ be an orthonormal basis of the linear span of
$\mathcal P$,
for the scalar product $\langle p_1,p_2\rangle = \int p_1 p_2 d\mu$.
Then
\begin{align}
        \E\left[  
                \sup_{p\in\mathcal P:\; \int (p-p_0^*)^2d\mu \le t_*^2}
                \frac 1 n \sum_{i=1}^n \epsilon_i (p-p_0^*)(Z_i)
        \right]
        &\le
        \E
        \sup_{\vtheta\in \R^d :\; \euclidnorms{\vtheta} \le t_*^2}
        \frac 1 n \sum_{i=1}^n \epsilon_i \sum_{j=1}^d e_j(X_i)
        \\
        &\le
        t_*
        \sqrt{
            \sum_{j=1}^d 
            \left(
                \frac 1 n \sum_{i=1}^n \epsilon_i e_j(X_i)
            \right)^2
        }
        \\
        &
        \le \frac{ t_* \sqrt{ \musupnorm{p_0}d}}{\sqrt n}
        = \frac{t_*^2}{16},
\end{align}
where we have used
the Cauchy-Schwarz inequality,
Jensen' inequality,
and that $\mathbb{E} e_j(X)^2 \le \musupnorm{p_0}$ for all $j=1,...,d$.
\end{proof}

\end{document}

\Cref{thm:iso-density,thm:isomorphic}
are similar because in each setting,
the associated empirical process is a linear function of the class members.
To be precise, \Cref{thm:iso-density} involves the linear process
\begin{equation}
    \left( 
        2 (P_n - P)(p - p_0^*)
    \right)_{p\in\mathcal P},
\end{equation}
while \Cref{thm:isomorphic} involves the random variable
\begin{equation}
    \left( 
        2 \vxi^T(\vf - \vf_0^*)
    \right)_{\vf\in K}.
\end{equation}
These processes are linear in the class members.
Nonlinearities appear for the Least Squares estimator
in regression with random design,
which is studied in \Cref{s:unknown-design}.

\section{Proofs}
\label{s:proofs}

\subsection{Strong convexity}
\begin{lemma}
    \label{lemma:strong-convexity}
    Let $(\mathcal Z,\mu)$ be a measurable
    and let $\mathcal P\subset L_2(\mu)$ be a set of measurable functions.
    Let $L:L_2(\mu) \rightarrow \R$ be a linear form
    and let 
    $H(p) = 
            \munorms{p} - L(p)
        $ for all $p\in\mathcal P$.
    Assume that $\mathcal P$ is convex and that
    there exists $\hat p\in\mathcal P$ such that
    $H(\hat p ) =
        \min_{p\in\mathcal P} 
        H(p)$.
    Then for all $p\in\mathcal P$,
    \begin{equation}
        H(\hat p) \le  H(p) - \munorms{p-\hat p}.
    \end{equation}
\end{lemma}
\begin{proof}
    Let $p\in\mathcal P$.
    By convexity, $t p+ (1-t)\hat p \in \mathcal P$
    for all $t\in[0,1]$.
    Let $Q(t) = H(t p+ (1-t)\hat p)$ for all $t\in[0,1]$.
    The function $Q$ is a polynomial function of degree 2,
    its exact Taylor expansion is given by
    \begin{equation}
        Q(t) = Q(0) + t Q'(0) + t^2 \munorms{p - \hat p}.
    \end{equation}
    The polynomial function $Q$ is minimized over $[0,1]$ at $t=0$ so
    that $Q'(0) \ge 0$.
    This yields $Q(1) \ge Q(0) + \munorms{p -\hat p}$
    which completes the proof.
\end{proof}

\end{document}

%% file: header.tex
\usepackage[T1]{fontenc}
\usepackage{lmodern}
\usepackage{amssymb,amsmath}
\usepackage{ifxetex,ifluatex}
\usepackage{fixltx2e} 
\IfFileExists{microtype.sty}{\usepackage{microtype}}{}
\IfFileExists{upquote.sty}{\usepackage{upquote}}{}
\ifnum 0\ifxetex 1\fi\ifluatex 1\fi=0 
  \usepackage[utf8]{inputenc}

\ifxetex
  \usepackage[setpagesize=false, 
              unicode=false, 
              xetex,
              hypertexnames=false
            ]{hyperref}
\else
  \usepackage[unicode=true,hypertexnames=false]{hyperref}
\fi
\hypersetup{breaklinks=true,
            linkcolor=magenta,
            citecolor=blue,
            colorlinks=true,
            pdfborder={0 0 0}}

\usepackage{enumitem}

\usepackage{amsthm}
\usepackage{thmtools}
\usepackage{mathtools}
\usepackage[numbers]{natbib}

\usepackage{hyperref}
\usepackage[noabbrev]{cleveref}
\usepackage{autonum}
\crefname{thm}{theorem}{theorems}
\crefname{lemma}{lemma}{lemmas}
\crefname{prop}{proposition}{propositions}
\crefname{assumption}{assumption}{assumptions}
\crefname{example}{examples}{examples}

\declaretheorem[name=Theorem]{thm}
\declaretheorem[name=Proposition, sibling=thm]{prop}
\declaretheorem[name=Lemma, sibling=thm]{lemma}

\crefname{setting}{setting}{settings}

%% file: macros.tex
\DeclareMathOperator*{\argmin}{argmin}

\DeclareMathOperator*{\rank}{rank}

\newcommand{\E}{\mathbb{E}}
\renewcommand{\epsilon}{\varepsilon}

\newcommand{\simplex}{\Lambda^M}
\newcommand{\design}{\mathbf X}

\newcommand{\proba}[1]{\mathbb{P}\left( #1 \right)}

\newcommand{\R}{\mathbf{R}}
\newcommand{\RM}{\R^M}
\newcommand{\Rn}{\R^n}

\newcommand{\munorms}[1]{\Vert #1 \Vert^2_{L_2(\mu)} }

\newcommand{\musupnorm}[1]{\Vert #1 \Vert_{L_\infty(\mu)} }

\newcommand{\euclidnorms}[1]{ | #1 |_2^2}
\newcommand{\euclidnorm}[1]{ | #1 |_2}
\newcommand{\onenorm}[1]{ | #1 |_1}

\newcommand{\hf}{{\boldsymbol{\hat f}}}
\newcommand{\vf}{{\boldsymbol{f}}}
\newcommand{\vmu}{{\boldsymbol{\mu}}}
\newcommand{\vtheta}{{\boldsymbol{\theta}}}
\newcommand{\vbeta}{{\boldsymbol{\beta}}}
\newcommand{\hbeta}{{\boldsymbol{\hat \beta}}}
\newcommand{\vv}{{\boldsymbol{v}}}
\newcommand{\ve}{{\boldsymbol{e}}}
\newcommand{\vzero}{{\boldsymbol{0}}}
\newcommand{\vx}{{\boldsymbol{x}}}
\newcommand{\htheta}{{\boldsymbol{\hat \theta}}}
\newcommand{\vu}{{\boldsymbol{u}}}
\newcommand{\vomega}{{\boldsymbol{\omega}}}

\newcommand{\vy}{\mathbf{y}}
\newcommand{\vxi}{{\boldsymbol{\xi}}}
\newcommand{\vg}{{\boldsymbol{g}}}

\newcommand{\that}{{\boldsymbol{\hat\theta}}}